\documentclass[12pt,reqno]{amsart}
\usepackage{amsmath,amsthm,amsfonts,amssymb,amscd,amstext}
\usepackage[ansinew]{inputenc}
\usepackage[dvips]{graphicx}
\usepackage{psfrag}
\usepackage{euscript}
\usepackage{a4wide}

\numberwithin{equation}{section}

\usepackage[colorlinks=true,linkcolor=red,urlcolor=black]{hyperref}

\newcommand{\qand}{\quad\text{and}\quad}

\theoremstyle{plain}
\newtheorem{maintheorem}{Theorem}

\newtheorem{theorem}{Theorem}[section]
\newtheorem{proposition}[theorem]{Proposition}
\newtheorem{corollary}[theorem]{Corollary}
\newtheorem{lemma}[theorem]{Lemma}

\theoremstyle{definition}
\newtheorem{remark}[theorem]{Remark}
\newtheorem{definition}{Definition}
\newtheorem{example}{Example}
\newtheorem{conjecture}{Conjecture}

\newcommand{\RR}{{\mathbb R}}

\newcommand{\la}{\lambda}

\renewcommand{\epsilon}{\varepsilon}

\newcommand{\diag}{\operatorname{diag}}
\newcommand{\trace}{\operatorname{tr}}

\newcommand{\sing}{\mathrm{Sing}}

\newcommand{\cC}{\EuScript{C}}

\newcommand{\J}{\EuScript{J}}

\newcommand{\wt}{\widetilde}

\begin{document}

\title[Dominated splitting for exterior powers]
{Dominated splitting for exterior powers and
  singular hyperbolicity}

%%%%%%%%%%%%%%%%%%%%%%%%%%%%%%%%%%%%%%%%%%%%%%%%%%%%%%%%%%%%%%%%%%%

\thanks{ V.A. was partially supported by CAPES, CNPq,
  PRONEX-Dyn.Syst. and FAPERJ (Brazil). L.S. was supported
  by CNPq and INCTMat-CAPES and by FAPESB-JCB0053/2013.  }

\renewcommand{\subjclassname}{\textup{2000} Mathematics
  Subject Classification}

%%%%%%%%%%%%%%%%%%%%%%%%%%%%%%%%%%%%%%%%%%%%%%%%%%%%%%%%%%%%%%%%%%%

\author{Vitor Araujo}
\address[V.A. \& L.S.]{Universidade Federal da Bahia,
Instituto de Matem\'atica\\
Av. Adhemar de Barros, S/N , Ondina,
40170-110 - Salvador-BA-Brazil}
\email{vitor.d.araujo@ufba.br \& lsalgado@ufba.br}

\author{Luciana Salgado} % \address[L.S.]{ Instituto de
%   Matem\'atica Pura e Aplicada - Estrada Dona Castorina,
%   110, Jardim Bot\^anico, 22460-320 Rio de Janeiro, Brazil }
% \email{lsalgado@impa.br}

%%%%%%%%%%%%%%%%%%%%%%%%%%%%%%%%%%%%%%%%%%%%%%%%%%%%%%%%%%%%%%%%%%%
\begin{abstract}
  We relate dominated splitting for a linear
  multiplicative cocyle with dominated splitting for
  the exterior powers of this cocycle. % This allows us to
  % show that singular-hyperbolicity is equivalent to
  % separability with respect to families of
  % infinitesimal Lyapunov functions both for the given
  % cocycle and the exterior square of this
  % cocycle.
  For a $C^1$ vector field $X$ on a $3$-manifold, we
  can obtain singular-hyperbolicity using only the tangent map $DX$
  of $X$ and a family of indefinite and non-degenerate quadratic forms
  without using the associated flow $X_t$ and its derivative $DX_t$.
  In this setting, we also improve a result from \cite{arsal2012a}.
  % We use the notion of infinitesimal Lyapunov function
  % applied to three-dimensional singular vector fields,
  % obtaining a characterization of singular-hyperbolicity
  % using the vector field $X$, its space derivative $DX$ and
  % a family of non-degenerate indefinite quadratic forms
  % only. This improves other results from the authors where
  % the same conclusion was obtained through the Linear
  % Poincar\'e flow, which is only defined away from
  % singularities.
  As a consequence, we show the existence of adapted metrics
  for singular-hyperbolic sets for three-dimensional $C^1$
  vector fields.
\end{abstract}

%%%%%%%%%%%%%%%%%%%%%%%%%%%%%%%%%%%%%%%%%%%%%%%%%%%%%%%%%%%%%%%%%%%

\date{\today}

\maketitle
\tableofcontents

\section{Introduction}
\label{sec:intro}

Let $M$ be a connected compact finite $n$-dimensional
manifold, $n \geq 3$, with or without boundary. We consider
a vector field $X$, such that $X$ is inwardly transverse to
the boundary $\partial M$, if $\partial M\neq\emptyset$.
The flow generated by $X$ is denoted by $\{X_t\}$.

A hyperbolic set for a flow $X_t$ on
%a finite dimensional Riemannian manifold
$M$ is a compact invariant set $\Gamma$
with a continuous splitting of the tangent bundle,
$T_\Gamma M= E^s\oplus E^X \oplus E^u$, where $E^X$ is the
direction of the vector field, for which the subbundles are
invariant under the derivative $DX_t$ of the flow $X_t$
\begin{align}\label{eq:hyp-splitting}
  DX_t\cdot E^*_x=E^*_{X_t(x)},\quad  x\in \Gamma, \quad t\in\RR,\quad *=s,X,u;
\end{align}
and $E^s$ is uniformly contracted by $DX_t$ and $E^u$ is
likewise expanded: there are $K,\lambda>0$ so that
\begin{align}\label{eq:Klambda-hyp}
  \|DX_t\mid_{E^s_x}\|\le K e^{-\lambda t},
  \quad
  \|(DX_t \mid_{E^u_x})^{-1}\|\le K e^{-\lambda t},
  \quad x\in \Gamma, \quad t\in\RR.
\end{align}
Very strong properties can be deduced from the existence of
such hyperbolic structure; see for
instance~\cite{Bo75,BR75,Sh87,KH95,robinson2004}.

Weaker notions of hyperbolicity (e.g. dominated
splitting, partial hyperbolicity, volume hyperbolicity,
sectional hyperbolicity, singular hyperbolicity etc)
have been developed to encompass larger classes of
systems beyond the uniformly hyperbolic ones;
see~\cite{BDV2004} and
specifically~\cite{viana2000i,AraPac2010,ArbSal2011} for singular
hyperbolicity and Lorenz-like attractors.

Proving the existence of some hyperbolic structure as
in~(\ref{eq:hyp-splitting}) and (\ref{eq:Klambda-hyp}),
is in general a non-trivial matter, even in its weaker
forms. We recall that the Lorenz attractor was shown to
exist through a computer assisted proof only very
recently in~\cite{Tu99} and, even more recently, in
\cite{HuntMack03} it was constructed a concrete example
of a mechanical system modeled by an Anosov flow.

The ``cone field technique'' is the usual way to prove
hyperbolicity even in some of its weaker forms; see e.g.
\cite{Aleks68,Aleks68a,Aleks69,robinson1999,Newhouse2004}.
Given a field of non-degenerate and indefinite quadratic
forms $\J:TM \to \RR$ with constant index, we define the
negative cone $C_-(x)$ as the set of vectors $v\in T_x M$
such that $\J(v) < 0$ and, analogously, we define the
positive cone $C_+(x)$ as the set of vectors $v\in T_x M$
such that $\J(v) > 0$.  Lewowicz used this notion in his
study of expansive homeomorphisms~\cite{lewow89} and
obtained an equivalence involving quadratic forms and
uniform hyperbolicity for diffeomorphisms. Indeed,
in~\cite{lewow89} it is proved that a diffeomorphism $f$ is
Anosov (that is, the whole manifold is a hyperbolic set) if,
and only, if there exists a field of non-degenerate and
indefinite quadratic forms $\J$ on the whole manifold $M$
such that the quadratic forms $ f^{\sharp}\J - \J $ are
everywhere positive definite, where $f^{\sharp}\J$ denotes
the pullback of the quadratic form by the derivative of $f$.

This idea was adapted for the study of Lyapunov exponents in
\cite{Wojtk85}, where a counterpart of the Lewowicz result
was obtained using the notion of $\J$-monotonicity, and was
also used to study stochastic properties of diffeomorphisms
in \cite{BurnKatok94}.

In \cite{arsal2012a} the authors extended these results
obtaining a necessary and sufficient condition for a maximal
invariant set $\Gamma$, possibly with singularities, of a
trapping region $U$, to be a partially hyperbolic set for a
$C^1$ flow $X_t$.

We briefly present these results in what follows with the
relevant definitions before stating the new results in this
paper.

Here we provide an alternative way to obtain singular
hyperbolicity for three-dimensional flows using the
same expression as in Proposition~\ref{pr:J-separated}
applied to the infinitesimal generator of the exterior
square $\wedge^2 DX_t$ of the cocycle $DX_t$.  This
infinitesimal generator can be explicitly calculated
through the infinitesimal generator $DX$ of the linear
multiplicative cocycle $DX_t$ associated to the vector
field $X$.

In a number of situations dealing with mathematical models
from the physical, engineering or social sciences, it is the
vector field that is given and not the flow. Thus we expect
that the results here presented to be useful to check some
weaker forms of hyperbolicity.

Indeed, we are able to explicitly prove that the
geometrical Lorenz attractor is singular-hyperbolic in
a straighforward way using this technique; see
Section~\ref{sec:comments-organiz-tex}.

As a consequence of these ideas we show the existence
of adapted metrics for singular-hyperbolic subsets for
general $C^1$ three-dimensional vector fields.

\subsection{Statement of preliminary results}
\label{sec:statem-prelim-result}

We recall that a \emph{trapping region} $U$ for a flow $X_t$ is an open
subset of the manifold $M$ which satisfies: $X_t(U)$ is
contained in $U$ for all $t>0$, and there exists $T>0$ such
that $\overline{X_t(U)} $ is contained in the interior of
$U$ for all $t>T$. We define $\Gamma(U)=\Gamma_X(U):=
\cap_{t>0}\overline {X_t(U)}$ to be the \emph{maximal
  positive invariant subset in the trapping region $U$}.

A \emph{singularity} for the vector field $X$
is a point $\sigma\in M$ such that $X(\sigma)=\vec0$ or,
equivalently, $X_t(\sigma)=\sigma$ for all $t \in \RR$. The
set formed by singularities is the \emph{singular set of
  $X$} denoted $\sing(X)$.  We say that a singularity is
  hyperbolic if the eigenvalues of the derivative
$DX(\sigma)$ of the vector field at the singularity $\sigma$
have nonzero real part

\begin{definition}\label{eq:domination}%\label{def1}
  %\cite[Definition 2.6]{MeMor06}
  A \emph{dominated splitting} over a compact invariant set $\Lambda$ of $X$
  is a continuous $DX_t$-invariant splitting $T_{\Lambda}M =
  E \oplus F$ with $E_x \neq \{0\}$, $F_x \neq \{0\}$ for
  every $x \in \Lambda$ and such that there are positive
  constants $K, \lambda$ satisfying
  \begin{align}\label{eq:def-dom-split}
    \|DX_t|_{E_x}\|\cdot\|DX_{-t}|_{F_{X_t(x)}}\|<Ke^{-\la
      t}, \ \textrm{for all} \ x \in \Lambda, \ \textrm{and
      all} \,\,t> 0.
  \end{align}
\end{definition}

A compact invariant set $\Lambda$ is said to be
\emph{partially hyperbolic} if it exhibits a dominated
splitting $T_{\Lambda}M = E \oplus F$ such that subbundle
$E$ is \emph{uniformly contracted}, i.e.,
there exists $C>0$ and $\lambda>0$ such that $\|DX_t|_{E_x}\|\leq Ce^{-\lambda}$ for $t\geq 0$.
In this case $F$ is the \emph{central subbundle} of $\Lambda$.
  Or else, we may replace uniform contraction along $E^s$ by
  uniform expansion along $F$ (the right hand side condition
  in (\ref{eq:Klambda-hyp}).
  %We present formal definitions
  %and statements of these results in Section
  %\ref{sec:prelim-definit}}

We say that a $DX_t$-invariant subbundle $F \subset
  T_{\Lambda}M$ is a \emph{sectionally expanding} subbundle
  if $\dim F_x \geq 2$ is constant for $x\in\Lambda$
  and there are positive constants $C , \lambda$ such that for every $x
    \in \Lambda$ and every two-dimensional linear subspace
    $L_x \subset F_x$ one has
    \begin{align}\label{eq:def-sec-exp}
    %\label{eq:sectional-expansion}
      \vert \det (DX_t \vert_{L_x})\vert > C e^{\la t},
      \textrm{ for all } t>0.
    \end{align}

\begin{definition}\label{sechypset} \cite[Definition
  2.7]{MeMor06} A \emph{sectional-hyperbolic set} is a
  partially hyperbolic set whose singularities are
  hyperbolic and the central subbundle is sectionally
  expanding.
\end{definition}

This is a particular case of the so called
\emph{singular hyperbolicity} whose definition we
recall now.  A $DX_t$-invariant subbundle $F \subset
T_{\Lambda}M$ is said to be a \emph{volume expanding}
if in the above condition \ref{eq:def-sec-exp}, we may
write
 \begin{align}\label{eq:def-sec-exp}
      \vert \det (DX_t \vert_{F_x})\vert > C e^{\la t},
      \textrm{ for all } t>0.
    \end{align}

\begin{definition}\label{singhypset} \cite[Definition
  1]{MPP99} A \emph{singular hyperbolic set} is a
  partially hyperbolic set whose singularities are
  hyperbolic and the central subbundle is volume
  expanding.
\end{definition}

%Sectional hyperbolicity was introduced in~\cite{MeMor06} as
%an extension of the notion of hyperbolicity encompassing
%invariant sets with singularities. It demands that the set
%be partially hyperbolic and that along the central bundle
%$E^c$ the area form restricted to any two-dimensional
%subspace be uniformly expanded.
We remark that, in the three-dimensional
case, these notions are equivalent.
%as \emph{singular hyperbolicity}
%as introduced in~\cite{MPP99}: the uniformly contracted
%direction is one-dimensional, and the area along the
%two-dimensional central direction is uniformly expanded:
%there are $K,\lambda>0$ such that
%\begin{align}\label{eq:sectional-expansion}
%  \vert \det (DX_t \mid {E^c_x})\vert > C e^{\la t}, \forall
%  t>0, x\in \Gamma.
%\end{align}
This is a feature of the Lorenz attractor as proved in
\cite{Tu99} and also a notion that extends
hyperbolicity for singular flows, because sectional
hyperbolic sets without singularities are hyperbolic;
see \cite{MPP04, AraPac2010}.

\begin{remark}
  \label{rmk:sec-exp-discrete}
  The properties of singular hyperbolicity can be expressed
  in the following equivalent forms; see \cite{AraPac2010}.
  There exists $T>0$ such that
  \begin{itemize}
  \item $\|DX_T\vert_{E^s_x}\|<\frac12$ for all $x\in\Gamma$
    (uniform contraction); and
  \item $|\det (DX_T\vert_{ E^c_x})|> 2$ for all
    $x\in\Gamma$.
  \end{itemize}
\end{remark}

%This means that we have a continuous
%splitting $T_\Gamma M =E^s\oplus E^c$ such that there are
%$K,\lambda>0$ satisfying the uniform contraction property in
%the left hand side of (\ref{eq:Klambda-hyp}) together with
%the domination condition
%  \begin{align}\label{eq:domination}
%    \|DX_t|_{E^s_x}\|\cdot\|DX_{-t}|_{E^c_{X_t(x)}}\|<Ke^{-\la
%      t}, \ \textrm{for all} \ x \in \Gamma, \ \textrm{and
%      all} \,\,t> 0
%  \end{align}
%  or else, we may replace uniform contraction along $E^s$ by
%  uniform expansion along $E^c$ (the right hand side condition
%  in (\ref{eq:Klambda-hyp}).  We present formal definitions
%  and statements of these results in Section
%  \ref{sec:prelim-definit}.

We say that a compact invariant subset $\Lambda$ is
\emph{non-trivial} if
\begin{itemize}
\item either $\Lambda$ does not contain singularities;
\item or $\Lambda$ contains at most finitely
many singularities, $\Lambda$ contains some
regular orbit and is connected.
\end{itemize}

\begin{theorem}\cite[Theorem A]{arsal2012a}
  \label{mthm:Jseparated-parthyp}
  A non-trivial compact invariant subset $\Gamma$ is a
  partially hyperbolic set for a flow $X_t$ if, and only if,
  there is a $C^1$ field $\J$ of non-degenerate and
  indefinite quadratic forms with constant index, equal to
  the dimension of the stable subspace of $\Gamma$, such
  that $X_t$ is a non-negative strictly $\J$-separated flow
  on a neighborhood $U$ of $\Gamma$.

  Moreover $E$ is a negative subspace, $F$ a positive
  subspace and the splitting can be made almost orthogonal.%, that is,
 % given $\epsilon>0$ we can find a field $\J$ of quadratic
 % forms as stated given by a family of self-adjoint
 % operators $J_x:T_xM\circlearrowleft$ such that
 % $|<J_x(v_-),v_+>|<\epsilon$ for all $v_-\in E^s_x, v_+\in
 % E^c_x, x\in \Gamma$ with $\J(v_\pm)=\pm1$ .
\end{theorem}
Here strict $\J$-separation corresponds to strict cone
invariance under the action of $DX_t$ and $\langle,\rangle$
is a Riemannian inner product in the ambient manifold. We
recall that the index of a field quadratic forms $\J$ on a
set $\Gamma$ is the dimension of the $\J$-negative space at
every tangent space $T_xM$ for $x\in U$.  Moreover, we say
that two subspaces $E$ and $F$ of a vector space are
\emph{almost orthogonal} if, given $\epsilon>0$, there
exists a inner product $\langle , \rangle$ so that $|\langle
u, v\rangle| < \epsilon$, for all $u \in E, v \in F$, with
$\| u \| = 1 = \| v\|$.

We note that the condition stated in
Theorem~\ref{mthm:Jseparated-parthyp} allows us to obtain
partial hyperbolicity checking a condition at every point of
the compact invariant set that depends only on the
tangent map $DX$ to the vector field $X$ together with a
family $\J$ of quadratic forms without using the flow
  $X_t$ or its derivative $DX_t$. This is akin to checking
the stability of singularity of a vector field using a
Lyapunov function.

In addition, we presented a criterion for partial
hyperbolicity through infinitesimal Lyapunov functions
based on the space derivative $DX$ of the vector field
$X$ only.  We assume that coordinates are chosen
locally adapted to $\J$ in such a way that
$\J(v)=\langle J_x(v),v\rangle, v\in T_xM, x\in U$, and
$J_x:T_xM\circlearrowright$ is a self-adjoint linear
operator having diagonal matrix with $\pm1$ entries
along the diagonal.

We say that a $C^1$ family $\J$ of indefinite and
non-degenerate quadratic forms is \emph{compatible}
with a continuous splitting $E_\Gamma\oplus
F_\Gamma=E_\Gamma$ of a vector bundle over some compact
subset $\Gamma$ if $E_x$ is a $\J$-negative subspace
and $F_x$ is a $\J$-positive subspace for all
$x\in\Gamma$.

\begin{proposition}\cite[Proposition 1.3]{arsal2012a}
  \label{pr:J-separated}
  A $\J$-non-negative vector field $X$ on $U$ is strictly
  $\J$-separated if, and only if, there exists a compatible
  family $\J_0$ of forms and there exists a function
  $\delta:U\to\RR$ such that the operator $\tilde J_{0,x}:=
  J_0\cdot DX(x)+DX(x)^*\cdot J_0$ satisfies
  \begin{align*}
    \tilde J_{0,x}-\delta(x)J_0 \quad\text{is positive
      definite}, \quad x\in U,
  \end{align*}
  where $DX(x)^*$ is the adjoint of $DX(x)$ with respect to
  the adapted inner product.
\end{proposition}

\begin{remark}\label{rmk:derivativeJ}
  The expression for $\tilde J_{0,x}$ in terms of $J_0$ and
  the infinitesimal generator of $DX_t$ is, in fact, the time
  derivative of $\J_0$ along the flow direction at the point
  $x$, which we denote $\partial_t J_0$; see item 1 of
  Proposition~\ref{pr:J-separated-tildeJ}. We keep this
  notation in what follows.
\end{remark}

The results leading to Theorem~\ref{mthm:Jseparated-parthyp}
and Proposition~\ref{pr:J-separated}, in the more general
case of linear multiplicative cocycles, were proved by the
authors in \cite{arsal2012a}, and then the general cocycle
can be replaced by the derivative cocycle $DX_t$ of the flow
$X_t$ with infinitesimal generator $DX$.

Building on this, in \cite[Corollary B and Proposition
1.4]{arsal2012a} it was obtained a necessary and
sufficient condition for the set $\Gamma$, possibly
with hyperbolic singularities, to be a
sectional-hyperbolic set for a $C^1$ flow $X_t$
% the existence of a $C^1$ family of non-degenerate
% quadratic forms $\J$ with constant index such that the
% flow is non-negative strictly $\J$-separated on $U$
% and, on every compact invariant subset $\Gamma_0\subset
% \Gamma \setminus \sing(X)$, the associated \emph{linear
%   Poincar\'e flow} is strictly $\J_0$-monotone for some
% field $\J_0$ of quadratic forms equivalent to $\J$ over
% $\Gamma_0$.
involving a stronger condition than the strict
$\J$-separation for the Linear Poincar\'e Flow of $X$
over all compact invariant subsets $\Gamma_0$ without
singularities of $\Gamma$.

%Sectional hyperbolicity was introduced in~\cite{MeMor06} as
%an extension of the notion of hyperbolicity encompassing
%invariant sets with singularities. It demands that the set
%be partially hyperbolic and that along the central bundle
%$E^c$ the area form restricted to any two-dimensional
%subspace be uniformly expanded.  In the three-dimensional
%case, this notion is known as \emph{singular hyperbolicity}
%as introduced in~\cite{MPP99}: the uniformly contracted
%direction is one-dimensional, and the area along the
%two-dimensional central direction is uniformly expanded:
%there are $K,\lambda>0$ such that
%\begin{align}\label{eq:sectional-expansion}
%  \vert \det (DX_t \mid {E^c_x})\vert > C e^{\la t}, \forall
%  t>0, x\in \Gamma.
%\end{align}
%This is a feature of the Lorenz attractor as proved in
%\cite{Tu99} and \emph{sectional hyperbolic sets without
%  singularities are hyperbolic}; see \cite{MPP04,
%  AraPac2010}.

A characterization of $\J$-monotonicity of the Linear
Poincar\'e Flow similar to the one in
Proposition~\ref{pr:J-separated} was also obtained in
\cite[Proposition 1.4]{arsal2012a} involving the space
derivative $DX$ of the field $X$, the field of forms $\J$
and the projection $\Pi\cdot DX$ on the normal bundle to $X$ away from
singularities. However, % checking for positive definiteness
% for some equivalent quadratic form is cumbersome and
dealing with the Linear Poincaré Flow near
singularities is prone to numerical instability and the
projection $\Pi \cdot DX$ does not extend to the
singularities.

\subsection{Statements of results}
\label{sec:statement-result}

Let $U$ be a trapping region for a $C^1$ vector field $X$ on
a compact, n-dimensional manifold $M$, which is
non-negative strictly $\J$-separated, and whose
singularities are hyperbolic in $U$. We write $\overline A$
for the topological closure of the set $A\subset M$ in what
follows. Let $\Gamma=\Gamma(U):=\overline{\cap_{t\in\RR}
  X_t(U)}$ be the maximal invariant set of $X$ in $U$.

Sectional-hyperbolicity deals with area expansion along any
two-dimensional subspace of a vector subbundle. It is then
natural to consider the linear multiplicative cocyle
$\wedge^2DX_t$ over the flow $X_t$ of $X$ on $U$, that is,
for any pair $u,v$ of vectors in $T_x M, x\in U$ and
$t\in\RR$ such that $X_t(x)\in U$ we set
\begin{align*}
  (\wedge^2 DX_t)\cdot(u\wedge v)=(DX_t\cdot
  u)\wedge(DX_t\cdot v),
\end{align*}
see \cite[Chapter 3, Section 2.3]{arnold-l-1998} or
\cite{Winitzki12} for more details and standard results
on exterior algebra and exterior products of linear
operator.

Given a partially hyperbolic splitting $T_\Gamma M=E_\Gamma\oplus
F_\Gamma$ over the compact $X_t$-invariant subset $\Gamma$, the
bundle of bivectors $\wedge^2T_\Gamma M$ admits also a partially
hyperbolic splitting, and $T_\Gamma M$ has a sectional
hyperbolic splitting if, and only if, $\wedge^2T_\Gamma M$
has a partial hyperbolic splitting of a specific kind. This
can in fact be extended to arbitrary $k$th exterior powers.

We note that if $E\oplus F$ is a $DX_t$-invariant
splitting of $T_\Gamma M$, with $\{e_1,\dots,e_\ell\}$
a family of basis for $E$ and $\{f_1,\dots,f_h\}$ a
family of basis for $F$, then $\widetilde F=\wedge^kF$
generated by $\{f_{i_1}\wedge\dots\wedge
f_{i_k}\}_{1\le i_1<\dots<i_k\le h}$ is naturally
$\wedge^kDX_t$-invariant by construction. In addition,
$\tilde E$ generated by $\{e_{i_1}\wedge\dots\wedge
e_{i_k}\}_{1\le i_1<\dots<i_k\le \ell}$ together with
all the exterior products of $i$ basis elements of $E$
with $j$ basis elements of $F$, where $i+j=k$ and
$i,j\ge1$, is also $\wedge^kDX_t$-invariant and,
moreover, $\widetilde E\oplus \widetilde F$ gives a
splitting of the $k$th exterior power $\wedge^k
T_\Gamma M$ of the subbundle $T_\Gamma M$.  One of our
results is the following.

\begin{maintheorem}\label{mthm:bivectparthyp}
  Let $T_\Gamma M=E_\Gamma\oplus F_\Gamma$ be a
  $DX_t$-invariant splitting over the compact
  $X_t$-invariant subset $\Gamma$ such that %$\dim E=s$,
  $\dim F=c\ge2$.  % For any given $2\le k\le c$,
  Let
  $\widetilde F=\wedge^c F$ be the
  $\wedge^cDX_t$-invariant subspace generated by the
  vectors of $F$ and $\tilde E$ be the
  $\wedge^cDX_t$-invariant subspace such that $\widetilde
  E\oplus\widetilde F$ is a splitting of the $c$th exterior
  power $\wedge^cT_\Gamma M$ of the subbundle $T_\Gamma M$.

  Then $E\oplus F$ is a dominated splitting if, and only if,
  $\widetilde E\oplus \widetilde F$ is a dominated splitting
  for $\wedge^cDX_t$.
\end{maintheorem}

As a consequence of this last result we can obtain sectional
hyperbolicity for three-dimensional flows using only the
second exterior power of the cocycle $DX_t$.
%condition over the linear Poincar\'e flow.

\begin{corollary}\label{cor:bivectparthyp}
  Assume, in the statement of Theorem
  \ref{mthm:bivectparthyp}, that $M$ has dimension $3$, $E$
  is uniformly contracted by $DX_t$ and that $c=2$.  Then
  $E\oplus F$ is a singular-hyperbolic splitting for $DX_t$
  if, and only if, $\widetilde E\oplus\widetilde F$ is a
  partially hyperbolic splitting for $\wedge^2DX_t$ such
  that $\widetilde F$ is uniformly expanded by $\wedge^2
  DX_t$.
\end{corollary}

\begin{remark}
  \label{rmk:discretecase}
  A similar statement to
  Theorem~\ref{mthm:bivectparthyp} is true for discrete
  dynamical systems, that is, replacing $DX_t$ in the
  statement of Theorem~\ref{mthm:bivectparthyp} and in
  (\ref{eq:hyp-splitting}),
  (\ref{eq:Klambda-hyp}), (\ref{eq:def-dom-split}) and
  also (\ref{eq:def-sec-exp}) by the tangent map $Df$ to
  a diffeomorphism $f$ of a compact manifold $M$.
\end{remark}

We note that it is not clear how to derive, from the
knowledge of $\J$ in a general situation, a field of
indefinite non-degenerate quadratic forms $\hat\J$
defined on $\wedge^2T_\Gamma M$ such that $\wedge^2 DX_t$ is
strictly $\hat\J$-separated; see
Example~\ref{ex:no-form} in the comments
Section~\ref{sec:comments-organiz-tex} below.

However, in a $3$-manifold, we show that
singular-hyperbolicity corresponds to strict $\J$-separation
for $DX_t$ together with strict $(-\J)$-separation for
$\wedge^2DX_t$ plus a condition on the rate function
$\delta$, so the same field of quadratic forms can be used
to obtain both partial hyperbolicity and
singular-hyperbolicity.

In a three-dimensional manifold, let $(u,v,w)$ be an
orthonormal base with positive orientation on $T_x M$ for a
given $x\in U$. Since we can identify $\wedge$ with the
cross-product $\times$, then for all $t\in\RR$ we can make the
identification
\begin{align}\label{eq:def-phi-t}
  \wedge^2DX_t \cdot w = (DX_tu)\times(DX_tv).
\end{align}
Now the meaning of Theorem~\ref{mthm:bivectparthyp} and
Corollary~\ref{cor:bivectparthyp} is clear: for an
orthogonal vector $w$ to the two-dimensional central
direction $F$, the variation of the size of $\wedge^2 DX_t
\cdot w$ corresponds to the variation of the area of the
parallelogram with sides $DX_t(x)u,DX_t(x)v$.  Hence, we
have uniform expansion of area along $F$ if, and only if,
$\wedge^2 DX_t$ uniformly expands the size of $w$.

The area under the function $\delta$ provided by
Proposition~\ref{pr:J-separated} allows us to detect
different dominated splittings with respect to linear
multiplicative cocycles on vector bundles (the natural
generalizations of the object $DX_t$ over $T_\Gamma M$ and
$\wedge^2DX_t$ over $\wedge^2 T_\Gamma M$); see
Proposition~\ref{pr:char-dom-split} in
Section~\ref{sec:prelim-definit}.

If $\wedge^2DX_t$ is strictly separated with respect to some
family $\J$ of quadratic forms, then there exists the
function $\delta_2$ as stated in
Proposition~\ref{pr:J-separated} with respect to the cocyle
$\wedge^2DX_t$.  We set
$$\widetilde\Delta_a^b(x) := \int_a^b\delta_2(X_s(x))\,ds$$
the area under the function $\delta_2:U\to\RR$ given by
Proposition~\ref{pr:J-separated} with respect to
$\wedge^2DX_t$ and its infinitesimal generator.

It is not difficult to see that this function is related to
$X$ and $\delta$ as follows: let $\delta:\Gamma\to\RR$ be
the function associated to $\J$ and $DX_t$, as given by
Proposition \ref{pr:J-separated}, then
$\delta_2=2\trace(DX)-\delta$, where $\trace(DX)$ represents
the trace of the linear operator $DX_x:T_xM\circlearrowleft,
x\in M$.

We recall that $\tilde \J = \partial_t \J$ is the time
derivative of $\J$ along the flow; see
Remark~\ref{rmk:derivativeJ}.

\begin{maintheorem}
  \label{mthm:sec-exp-3d}
  Suppose that $X$ is three-dimensional vector field on M
  which is non-negative strictly
  $\J$-separated over a non-trivial subset $\Gamma$, where
  $\J$ has index $1$.
Then
\begin{enumerate}
\item $\wedge^2DX_t$ is strictly $(-\J)$-separated;
\item $\Gamma$ is a singular hyperbolic set if either one of the
following properties is true
\begin{enumerate}
\item $\widetilde\Delta_0^t(x)\xrightarrow[t\to+\infty]{}-\infty$
    for all $x\in\Gamma$.%$\widetilde\Delta_0^t$ satisfies condition (1) of  Theorem~\ref{pr:char-dom-split}.
\item $\tilde \J-2\trace(DX)\J>0$ on $\Gamma$.
\end{enumerate}
%where $\delta_2=2\trace(DX)-\delta$.
\end{enumerate}
\end{maintheorem}

This result provides useful sufficient conditions for
a three-dimensional vector field to be singular hyperbolic,
using only one family of quadratic forms $\J$ and its space
derivative $DX$, avoiding the need to check cone invariance
and contraction/expansion conditions for the flow $X_t$
generated by $X$ on a neighborhood of $\Gamma$; see the
examples in Section~\ref{sec:comments-organiz-tex} below.

To geometrically understand Theorem~\ref{mthm:sec-exp-3d},
let us consider a singular hyperbolic compact set $\Gamma$
with partial hyperbolic splitting $T_\Gamma M=
E_\Gamma^s\oplus E_\Gamma^c$. Following \cite{Goum07}, we
can obtain a smooth Riemannian adapted metric to the partial
hyperbolic splitting so that the decomposition becomes
almost orthogonal.  In this setting, it is clear that at
each point $x\in \Gamma$ and with respect to this metric we
have
\begin{align}\label{eq:close}
  \hat E_\Gamma^u:=(E^c_A)^\perp \approx E^s_\Gamma
  \qand
  \hat E_\Gamma^c=(E^s_\Gamma)^\perp \approx E^c_\Gamma
\end{align}
where $\approx$ means that the subbundles are inside a cone
of small width centered at one of them.

Hence, by definition of $\wedge^2 DX_t$, the decomposition
$\hat E^c_\Gamma\oplus \hat E^u_\Gamma$ is also
$\wedge^2DX_t$-invariant. In addition, $\wedge^2DX_t$
expands the length along the $\hat E^u_\Gamma$ direction, due to
area expansion along the $E^c_\Gamma$ direction under the action of
$DX_t$. Moreover, $\hat E^c_\Gamma$ is dominated by $\hat E^u_\Gamma$
since the area along $\hat E^c_\Gamma$ should be contracted under
the action of $\wedge^2DX_t$. This provides a partial
hyperbolic splitting for $\wedge^2DX_t$.

Therefore, by Theorem~\ref{mthm:Jseparated-parthyp}, there
exists some family of quadratic forms such that $\wedge^2
DX_t$ is strictly separated. But to arrive at the right
expansion and domination relations, we should have that
$\hat E^u$ is now inside the positive cone, and $\hat E^c_\Gamma$
inside the negative cone, so that $\hat E^u_\Gamma$ dominates $\hat
E^c_\Gamma$. By~(\ref{eq:close}) this can precisely be achieved by
taking $(-\J)$ as our family of quadratic forms.

The definition of hyperbolicity is clearly independent
of the Riemannian metric on the manifold $M$.  By a
recent result from \cite{Goum07}, there exists an
adapted metric on $M$ for $X_t$, which means that the
constants in the above
expressions~(\ref{eq:Klambda-hyp}) and
(\ref{eq:def-dom-split}) become $1$.

%\begin{theorem}
%  \label{thm:phyp-adapted-metric}
%  Let $\Gamma$ be a compact invariant set for a $C^1$
%  vector field $X$ and let $E$ be a vector bundle over
%  a neighborhood $U$ of $\Gamma$.  Suppose that
%  $T_{\Gamma}M = E\oplus F$ is a dominated splitting
%  for a linear multiplicative cocycle $A_t(x)$ over
%  $E$.  There exists a Riemannian metric
%  $<<\cdot,\cdot>>$ inducing a norm $|\cdot|$ on $E$
%  such that there exists $\lambda>0$ satisfying
%  \begin{align}\label{eq:adapted-domination}
%    |A_t\mid E_x|\cdot\big|(A_t\mid F_x)^{-1}|
%    \le e^{-\lambda t} \quad
%     x\in\Gamma, t>0.
%  \end{align}
%  If the splitting is partially hyperbolic with $E$
%  uniformly contracted, then there are
%  $<<\cdot,\cdot>>$ and $\lambda>0$ such that
%  \begin{align}\label{eq:adapted-contraction}
%     |A_t\mid E^s_x|\le e^{-\lambda t}, \quad
%     x\in\Gamma, t>0.
%  \end{align}
%  (Analogously for the partially hyperbolic case with
%  expanding $F$ by reversing the flow direction.)
%  Moreover, it is possible to extend such metric to a
%  neighborhood $V\subset U$ of $\Gamma$ such that, in
%  the adapted metric, the bundles $E,F$ over $\Gamma$
%  are almost orthogonal: given $\epsilon>0$ it is
%  possible to construct such metrics satisfying
%  $|<<u,v>>|<\epsilon$ for all $u\in E_x, v\in F_x$
%  with $|u|=|v|=1$.
%\end{theorem}
\begin{definition}
\label{adapmetric3d}
We say a Riemannian metric $\langle\cdot,\cdot\rangle$  \emph{adapted to a singular hyperbolic splitting} $T\Gamma = E \oplus F$
if it induces a norm $|\cdot|$ such that there exists $\lambda>0$ satisfying for all
  $x\in \Gamma$ and $t>0$ simultaneously
    \begin{align*}
    |DX_t\mid_{E_x}|\cdot\big|(DX_t\mid_{F_x})^{-1}|
    \le e^{-\lambda t}, \ % \quad\text{and}\quad
     |DX_t\mid_{E_x}|\le e^{-\lambda t}
\quad\text{and}\quad \vert \det (DX_t \mid_{F_x})\vert \ge e^{\lambda t}.
\end{align*}
We call it \emph{singular adapted metric}, for simplicity.
\end{definition}

This result is used in the proof of
Theorem~\ref{mthm:Jseparated-parthyp}. As a consequence
of the proof of Theorem~\ref{mthm:sec-exp-3d} we show
that for $C^1$ flows having a singular-hyperbolic set
$\Gamma$ there exists a metric adapted to the partial
hyperbolicity and the area expansion, as follows.

\begin{maintheorem}
  \label{mthm:adapted3d}
  Let $\Gamma$ be a singular-hyperbolic
  set for a $C^1$ three-dimensional vector field $X$.
  %, that is, $T_\Gamma M=E\oplus F$ is a partially hyperbolic
  %splitting satisfying (\ref{eq:domination}) together with
  %uniform contraction along $E$ and area expansion
  %(\ref{eq:sectional-expansion}) along $F$.
  Then $\Gamma$ admits a singular adapted metric.
 % $\langle\cdot,\cdot\rangle$ inducing a norm $|\cdot|$ on $E$
 % such that there exists $\lambda>0$ satisfying for all
 % $x\in \Gamma$ and $t>0$ simultaneously
 %   \begin{align*}
 %   |DX_t\mid E_x|\cdot\big|(DX_t\mid F_x)^{-1}|
 %   \le e^{-\lambda t} \quad\text{and}\quad
 %    |DX_t\mid E_x|\le e^{-\lambda t}
 % \end{align*}
%and also $\vert \det (DX_t \mid {F_x})\vert \ge e^{\lambda t}.$
\end{maintheorem}

%%%%%%%%%%%%%%%%%%%%%%%%%%%%%%%%%%%%%%%%%%%%%%%%%%%%%%%%%%%%%%%%%%

We complete the introduction with some conjectures about the
adapted metric as in Theorem \ref{mthm:adapted3d}.
% Indeed, we suppose that this should be more general.

\begin{conjecture}\label{conj:existsad}
  There exists an adapted metric for all
  sectional-hyperbolic sets for any $C^1$ vector field in
  any dimension.
\end{conjecture}

Moreover, we should extend this to more general notions of
sectional-expansion of area.

\begin{conjecture}\label{conj:existskexp}
  There exists an adapted metric for all compact invariant
  subsets of a $C^1$ vector field $X$ on a manifold $M$,
  which are partially hyperbolic with splitting $E\oplus F$,
  $E$ uniformly contracted and all $k$-subspaces of $F$ are
  volume-expanding.
 \end{conjecture}
 In terms of exterior powers, the last condition on
 Conjecture~\ref{conj:existskexp} means that there are
 $C,\lambda>0$ such that
  $$\|(\wedge^k DX_t)\cdot(v_1\wedge\dots\wedge
  v_k)\|\ge C e^{\lambda t},$$ for all $t>0$ and every
  $k$-frame $v_1,\dots, v_k$ inside $F$, with $2 \leq k
  \leq \dim(F)$.  In other words, we can find an
  adapted Riemannian metric for $M$ whose naturally
  induced norm in the $k$-exterior product of $TM$
  satisfies the above inequality for some $\lambda>0$
  and $C=1$.

This should also be true for discrete dynamical systems.

\begin{conjecture}
  There exists an adapted metric for all compact
  invariant subsets of a $C^1$ diffeomorphism admitting
  a partially hyperbolic splitting $E\oplus F$, where
  $E$ is uniformly contracted and all $k$-subspaces of
  $F$ are volume-expanding.
\end{conjecture}

%%%%%%%%%%%%%%%%%%%%%%%%%%%%%%%%%%%%%%%%%%%%%%%%%%%%%%%%%%

\subsection{Organization of the text}
\label{sec:organization-text}

The main definitions and results on linear multiplicative
cocycles needed for our arguments here are presented in
Section~\ref{sec:prelim-definit}.

The proofs of Theorem~\ref{mthm:bivectparthyp},
Theorem~\ref{mthm:sec-exp-3d} and
Theorem~\ref{mthm:adapted3d}, presented in
Subsections~\ref{sec:dominat-splitt-exter},
\ref{sec:three-dimens-case} and
\ref{sec:existence-adapted-in} respectively, depend on
several results about $\J$-separation for linear
multiplicative cocycles given in
Section~\ref{sec:prelim-definit}.
% In Section~\ref{sec:conj}, we present some conjectures
% related to Theorem \ref{mthm:adapted3d}.
Finally, we present some examples of application in
Section~\ref{sec:comments-organiz-tex}.

\subsection*{Acknowledgments}
A significant part of the work was developed in L.S. thesis
\cite{luciana-tese} at Instituto de Mate\-m\'atica da
Universidade Federal do Rio de Janeiro (UFRJ), postdoc
research at IMPA and Instituto de Matem\'atica da
Universidade Federal da Bahia.

%%%%%%%%%%%%%%%%%%%%%%%%%%%%%%%%%%%%%%%%%%%%%%%%%%%%%%%%%

\section{Some definitions and useful results}
\label{sec:prelim-definit}

\subsection{Fields of quadratic forms, positive and negative
  cones}
\label{sec:fields-quadrat-forms}

Let $E_U$ be a finite dimensional vector bundle with
inner product $\langle\cdot,\cdot\rangle$ and
base given by the trapping region $U\subset M$. Let
$\J:E_U\to\RR$ be a continuous family of quadratic
forms $\J_x:E_x\to\RR$ which are non-degenerate and
have index $0<q<\dim(E)=n$.  The index $q$ of $\J$
means that the maximal dimension of subspaces of
non-positive vectors is $q$. Using the inner product,
we can represent $\J$ by a family of self-adjoint
operators $J_x:E_x\circlearrowleft$ as $\J_x(v)=\langle
J_x(v),v\rangle, v\in E_x, x\in U$.

We also assume that $(\J_x)_{x\in U}$ is continuously
differentiable along the flow.  The continuity assumption on
$\J$ means that for every continuous section $Z$ of $E_U$
the map $U\ni x\mapsto \J(Z(x))\in\RR$ is
continuous. The $C^1$ assumption on $\J$ along the flow
means that the map $\RR\ni t\mapsto \J_{X_t(x)} (Z(X_t(x)))\in \RR$
is continuously differentiable for all $x\in U$ and each
$C^1$ section $Z$ of $E_U$.

Using Lagrange diagonalization of a quadratic form, it is
easy to see that the choice of basis to diagonalize $\J_y$
depends smoothly on $y$ if the family $(\J_x)_{x\in U}$ is
smooth, for all $y$ close enough to a given $x$. Therefore,
choosing a basis for $T_x$ adapted to $\J_x$ at each $x\in
U$, we can assume that locally our forms are given by
$\langle J_x(v),v\rangle$ with $J_x$ a diagonal matrix whose
entries belong to $\{\pm1\}$, $J_x^*=J_x$, $J_x^2=I$ and the
basis vectors depend as smooth on $x$ as the family of forms
$(\J_x)_x$.

We let $\cC_\pm=\{C_\pm(x)\}_{x\in U}$ be the family of
positive and negative cones associated to $\J$
\begin{align*}
  C_\pm(x):=\{0\}\cup\{v\in E_x: \pm\J_x(v)>0\}  \quad x\in U
\end{align*}
and also let $\cC_0=\{C_0(x)\}_{x\in U}$ be the correspoing
family of zero vectors $C_0(x)=\J_x^{-1}(\{0\})$ for all
$x\in U$.

\subsection{Linear multiplicative cocycles over flows}
\label{sec:linear-multipl-cocyc}

Let $A:E\times\RR\to E$ be a smooth map given by a collection of linear
bijections
\begin{align*}
  A_t(x): E_x\to E_{X_t(x)}, \quad x\in M, t\in\RR,
\end{align*}
where $M$ is the base space of the finite dimensional vector
bundle $E$, satisfying the cocycle property
\begin{align*}
  A_0(x)=Id, \quad A_{t+s}(x)=A_t(X_s(x))\circ A_s(x), \quad x\in M, t,s\in\RR,
\end{align*}
with $\{X_t\}_{t\in\RR}$ a smooth flow over $M$.  We note
that for each fixed $t>0$ the map $A_t: E\to E, v_x\in E_x
\mapsto A_t(x)\cdot v_x\in E_{X_t(x)}$ is an automorphism of
the vector bundle $E$.

The natural example of a linear multiplicative cocycle over
a smooth flow $X_t$ on a manifold is the derivative cocycle
$A_t(x)=DX_t(x)$ on the tangent bundle $TM$ of a finite
dimensional compact manifold $M$.

The following definitions are fundamental to state our main
result.

\begin{definition}
\label{def:J-separated}
Given a continuous field of non-degenerate quadratic forms
$\J$ with constant index on the positively invariant open
subset $U$ for the flow $X_t$, we say that the cocycle
$A_t(x)$ over $X_t$ is
\begin{itemize}
\item $\J$-\emph{separated} if $A_t(x)(C_+(x))\subset
  C_+(X_t(x))$, for all $t>0$ and $x\in U$ (simple cone invariance);
\item \emph{strictly $\J$-separated} if $A_t(x)(C_+(x)\cup
  C_0(x))\subset C_+(X_t(x))$, for all $t>0$ and $x\in U$
  (strict cone invariance).
% \item $\J$-\emph{monotone} if $\J_{X_t(x)}(A_t(x)v)\ge \J_x(v)$, for each $v\in
%   T_xM\setminus\{\vec0\}$ and $t>0$;
% \item \emph{strictly $\J$-monotone} if $\partial_t\big(\J_{X_t(x)}(A_t(x)v)\big)\mid_{t=0}>0$,
%   for all $v\in T_xM\setminus\{0\}$, $t>0$ and $x\in U$;
% \item $\J$-\emph{isometry} if $\J_{X_t(x)}(A_t(x)v) = \J_x(v)$, for each $v\in T_xM$ and $x\in U$.
\end{itemize}
\end{definition}
% Thus, $\J$-separation corresponds to simple cone invariance
% and strict $\J$-separation corresponds to strict cone
% invariance under the action of $A_t(x)$.

We say that the flow $X_t$ is (strictly)
$\J$-\emph{separated} on $U$ if $DX_t(x)$ is (strictly)
$\J$-\emph{separated} on $T_UM$.
% Analogously, the flow of $X$ on $U$ is (strictly) $\J$-\emph{monotone} if
% $DX_t(x)$ is (strictly) $\J$-\emph{monotone}.

\begin{remark}\label{rmk:J-separated-C-}
  If a flow is strictly $\J$-separated, then for $v\in T_xM$
  such that $\J_x(v)\le0$ we have
  $\J_{X_{-t}(x)}(DX_{-t}(v))<0$ for all $t>0$ and $x$ such
  that $X_{-s}(x)\in U$ for every $s\in[-t,0]$. Indeed,
  otherwise $\J_{X_{-t}(x)}(DX_{-t}(v))\ge0$ would imply
  $\J_x(v)=\J_x\big(DX_t(DX_{-t}(v))\big)>0$, contradicting
  the assumption that $v$ was a non-positive vector.

  This means that a flow $X_t$ is strictly
    $\J$-separated if, and only if, its time reversal
    $X_{-t}$ is strictly $(-\J)$-separated. % The same remark
  % applies in the case of a cocycle.
\end{remark}

A vector field $X$ is $\J$-\emph{non-negative} on $U$ if
$\J(X(x))\ge0$ for all $x\in U$, and
$\J$-\emph{non-positive} on $U$ if $\J(X(x))\leq 0$ for all
$x\in U$. When the quadratic form used in the context is
clear, we will simply say that $X$ is non-negative or
non-positive.

%%%%%%%%%%%%%%%%%%%%%%%%%%%%%%%%%%%%%%%%%%%%%%%%%%%%%%%%

\subsection{Properties of $\J$-separated linear multiplicative cocycles}
\label{sec:propert-j-separat}

We present some useful properties about $\J$-separated
linear cocycles whose proofs can be found in
\cite{arsal2012a}.

Let $A_t(x)$ be a linear multiplicative cocycle over
$X_t$. We define the infinitesimal generator of $A_t(x)$ by
\begin{align}\label{eq:infinitesimal-gen}
  D(x):=\lim_{t\to0}\frac{A_t(x)-Id}t.
\end{align}

The following is the basis of our arguments leading to
Theorem~\ref{mthm:Jseparated-parthyp}.

The area under the function $\delta$ provided by
Proposition~\ref{pr:J-separated-tildeJ} allows us to detect
different dominated splittings with respect to linear
multiplicative cocycles on vector bundles (Proposition~\ref{pr:char-dom-split}).
For this, define the function
    \begin{align}\label{eq:delta-area}
     \Delta_a^b(x):=\int_a^b\delta(X_s(x))\,ds, \quad
     x\in\Gamma, a,b\in\RR.
    \end{align}

\begin{proposition}\cite[Proposition 2.7]{arsal2012a}
  \label{pr:J-separated-tildeJ}
  Let $A_t(x)$ be a cocycle over $X_t$ defined on an open
  subset $U$ and $D(x)$ its infinitesimal generator. Then
  \begin{enumerate}
  \item $\tilde\J(v)=\partial_t \J(A_t(x)v) =
    \langle \tilde J_{X_t(x)} A_t(x)v,A_t(x)v\rangle$ for
    all $v\in E_x$ and $x\in U$, where
    \begin{align}\label{eq:J-separated-tildeJ}
      \tilde J_x:= J\cdot D(x) + D(x)^* \cdot J
    \end{align}
    and $D(x)^*$ denotes the adjoint of the linear map
    $D(x):E_x\to E_x$ with respect to the adapted inner
    product at $x$;
  \item the cocycle $A_t(x)$ is $\J$-separated if, and only
    if, there exists a neighborhood $V$ of $\Lambda$,
    $V\subset U$ and a function $\delta:V\to\RR$ such that
    \begin{align}\label{eq:J-ge}
      \tilde \J_x\ge\delta(x)\J_x
      \quad\text{for all}\quad x\in V.
    \end{align}
    In particular we get $\partial_t\log|\J(A_t(x)v)|\ge
    \delta(X_t(x))$, $v\in E_x, x\in V, t\ge0$;
  \item if the inequalities in the previous item are strict,
    then the cocycle $A_t(x)$ is strictly
    $\J$-separated. Reciprocally, if $A_t(x)$ is strictly
    $\J$-separated, then there exists a compatible family
    $\J_0$ of forms on $V$ satisfying the strict
    inequalities of item (2).
  \item For a $\J$-separated cocycle $A_t(x)$, we have
    $\frac{|\J(A_{t_2}(x)v)|}{|\J(A_{t_1}(x)v)|}\ge \exp
    \Delta_{t_1}^{t_2}(x)$ for all $v\in E_x$ and reals
    $t_1<t_2$ so that $\J(A_t(x)v)\neq0$ for all $t_1\le
    t\le t_2$, where $\Delta_{t_1}^{t_2}(x)$ was defined
    in~(\ref{eq:delta-area}).
  % \item if $A_t(x)$ is $\J$-separated and $x\in\Gamma(U), v\in
  %   C_+(x)$ and $w\in C_-(x)$ are non-zero vectors, then for
  %   every $t>0$ such that $A_s(x)w\in C_-(X_s(x))$ for all
  %   $0<s<t$
  % \begin{align}\label{eq:Jquotient-upper}
  %   \frac{|\J(A_t(x)w)|}{\J(A_t(x)v)}
  %   \le
  %   \frac{|\J(w)|}{\J(v)} \exp \big(2\Delta_0^t(x)).
  % \end{align}
\item we can bound $\delta$ at every $x\in\Gamma$ by
    $\inf_{v\in C_+(x)}\frac{\tilde\J(v)}{\J(v)}
    \le\delta(x)\le
    \sup_{v\in C_-(x)}\frac{\tilde\J(v)}{\J(v)}.$
  \end{enumerate}
\end{proposition}

\begin{remark}
  \label{rmk:strictly-J-separated}
  We stress that the necessary and sufficient condition in
  items (2-3) of Proposition~\ref{pr:J-separated-tildeJ},
  for (strict) $\J$-separation, shows that a cocycle
  $A_t(x)$ is (strictly) $\J$-separated if, and only if, its
  inverse $A_{-t}(x)$ is (strictly) $(-\J)$-separated.
\end{remark}

\begin{remark}
  \label{rmk:Jexp-ineq}
  Item (2) above of Proposition~\ref{pr:J-separated-tildeJ}
  shows that $\delta$ is a measure of the ``minimal
  instantaneous expansion rate'' of $|\J\circ A_t(x)|$.%  on
  % positive vectors; and item (5) shows in addition that
  % $\delta$ is also a bound for the ``instantaneous variation
  % of the ratio'' between $|\J\circ A_t(x)|$ on negative and
  % positive vectors.
\end{remark}

\begin{proposition}\cite[Theorem 2.23]{arsal2012a}
  \label{pr:char-dom-split}
  Let $\Gamma$ be a compact invariant set for $X_t$
  admitting a dominated splitting $E_\Gamma= F_-\oplus F_+$
  for $A_t(x)$, a linear multiplicative cocycle over
  $\Gamma$ with values in $E$. Let $\J$ be a $C^1$ family
    of indefinite quadratic forms such that $A_t(x)$ is
    strictly $\J$-separated. Then
  \begin{enumerate}
  \item $F_-\oplus F_+$ is partially hyperbolic with % $F_-$
    % not uniformly contracting and
    $F_+$ uniformly expanding
    if % , and only if,
    $\Delta_0^t(x)\xrightarrow[t\to+\infty]{} +\infty$
    for all $x\in\Gamma$.
  \item $F_-\oplus F_+$ is partially hyperbolic with $F_-$
    uniformly contracting %and $F_+$ not uniformly expanding
    if %, and only if,
    $\Delta_0^t(x)\xrightarrow[t\to+\infty]{}-\infty$
    for all $x\in\Gamma$.
  \item $F_-\oplus F_+$ is uniformly hyperbolic
    if, and only if, there exists a compatible family $\J_0$
    of quadratic forms in a neighborhood of $\Gamma$ such
    that $\J_0'(v)>0$ for all $v\in E_x$ and all $x\in\Gamma$.
  \end{enumerate}
\end{proposition}
Above we write $\tilde\J(v)= <\tilde J_x v, v>$, where
$\tilde J_x$ is given in Proposition~\ref{pr:J-separated},
that is, $\tilde\J(v)$ is the time derivative of $\J$ under the
action of the flow.

We use Proposition~\ref{pr:char-dom-split} to obtain
sufficient conditions for a flow $X_t$ on $3$-manifold
to have a $\wedge^2 DX_t$-invariant one-dimensional
uniformly expanding direction orthogonal to the
two-dimensional center-unstable bundle.

\begin{remark}
  \label{rmk:correction}
  The proof of \cite[Theorem 2.23]{arsal2012a} on which
  Proposition~\pageref{pr:J-separated-tildeJ} is based, is
  correct for the sufficient conditions in items (1-2) (but
  the proof of necessity in items (1-2) in \cite{arsal2012a}
  is wrong).
\end{remark}

%%%%%%%%%%%%%%%%%%%%%%%%%%%%%%%%%%%%%%%%%%%%%%%%%%%%%%%%
\section{The exterior square of the cocycle}
\label{sec:exteri-square-cocycl}

%Now we are going to prove Theorem \ref{mthm:Jseparated-parthyp}.

We consider the action of the cocycle $DX_t(x)$ on
$k$-vector first and bivectors later, that is, the exterior
square $\wedge^kDX_t$ of the cocycle acting on
$\wedge^kT_\Gamma M$ with $k>2$ and then $k=2$, to deduce
Theorem~\ref{mthm:bivectparthyp} and Corollary~\ref{cor:bivectparthyp} first
and then prove
Theorem~\ref{mthm:sec-exp-3d} and
Theorem~\ref{mthm:adapted3d}.

\subsection{Dominated splitting and the exterior cocycle}
\label{sec:dominat-splitt-exter}

We denote by $\|\cdot\|$ the standard norm on bivectors
induced by the Riemannian norm of $M$; see,
e.g. \cite{arnold-l-1998}. We write $m=\dim M$.

\begin{proof}[Proof of Theorem~\ref{mthm:bivectparthyp}]
  We assume that $T_\Gamma M$ admits a dominated splitting
  $E_\Gamma\oplus F_\Gamma$ with $\dim E_\Gamma=s$ and $\dim
  F_\Gamma=c$. So there
  exists $\eta>0$ such that, for any $X_t$-invariant
  probability measure $\mu$ supported on $\Gamma$, the
  Lyapunov exponents of $DX_t$ with respect to $\mu$ are
  (repeated according to multiplicity)
  $\lambda_1\le\dots\le\lambda_s\le\lambda_{s+1}\le\dots\le\lambda_m$
  and satisfy $\lambda_{s+1}-\lambda_s>\eta$.

  The Lyapunov exponents of $\wedge^cDX_t$ are given by
  $\{\lambda_{i_1}+\dots+\lambda_{i_c}\}_{1\le i_1 <\dots<\i_c \le s+c}$; see
  e.g. \cite[Chapter 3]{arnold-l-1998}. Hence, for $\tilde
  F=\wedge^cDX_t$ and $\tilde E$ as in the statement of
  Theorem \pageref{mthm:bivectparthyp}, we have that
  \begin{align*}
    (\lambda_{i_1}+\dots+\lambda_{i_n})+(\lambda_{h_1}+\dots+\lambda_{h_m})+n\eta
    <
    \lambda_{s+1}+\dots+\lambda_{s+c}
  \end{align*}
  for all $1\le i_1<\dots<i_n\le s$, $s+1\le h_1<\dots<h_m\le c$ with
  $m+n=c,m,n\ge1$.
  This implies that for $\mu$-almost every $x\in\Gamma$
  \begin{align}\label{eq:sumliap}
    \lim_{t\to+\infty}\frac1t\log&\big(\|\wedge^cDX_t\mid
    \widetilde E_x\|\cdot\|(\wedge^cDX_t\mid \widetilde
    F_x)^{-1}\|\big) \nonumber
    \\
    &= \lambda_{\max\{s-c,1\}}+\dots+\lambda_{\max\{s-c,1\}+c} -
    (\lambda_{s+1}+\dots+\lambda_{s+c}) \le -\eta,
  \end{align}
that is, the maximum rate of expansion along $\widetilde E$
minus the minimum rate of expansion along $\widetilde F$.

  We now set $f_t(x)=\log\big(\|\wedge^cDX_t\mid \widetilde
  E_x\|\cdot\|(\wedge^cDX_t\mid \widetilde
  F_x)^{-1}\|\big)$ and, since we obtain~(\ref{eq:sumliap})
  for an arbitrary $X_t$-invariant probability measure, we
  can apply the following result which is an improvement from \cite[Proposition
  3.4]{arbieto2010}.

  \begin{lemma}\cite[Corollary 4.2]{ArbSal2011}
    \label{le:sectional-arbieto}
   Let $\{t\mapsto f_t:S\to \mathbb{R}\}_{t\in \mathbb{R}}$ be a continuous family of continuous functions
   which is subadditive and suppose that $\int \wt{f}(x) d\mu < 0$
   for every $\mu\in \mathcal{M}_X$, with $\wt{f}(x):=\lim\limits_{t\to+\infty}\frac{1}{t}f_t(x)$.
   Then there exist a $T > 0$ and a constant $\lambda < 0$
   such that for every $x\in S$ and every $t \geq T$:
   $$f_t(x) \leq  \lambda t.$$
%   Let $\{f_t:\Gamma\to\RR\}_{t\in\RR}$ be a continuous
%    family of subadditive continuous functions such that
%    there exists $\eta>0$ satisfying $\liminf_{t\to+\infty}
%    f_t(x)/ t <-\eta<0$ for $x$ in a subset $Z$ of $\Gamma$
%    with measure one with respect to every invariant
%    probability measure on $\Gamma$. Then there exists a
%    constant $C>0$ such that for all $t\ge0$ we have $\exp
%    f_t(x)\le C e^{-\eta t/2}$ for all $x\in\Gamma$.
  \end{lemma}

  We thus have $f_t(x)\le \kappa -\eta t, t\ge0,
  x\in\Lambda$ for a constant $\kappa>0$, as required for a
  dominated splitting with respect to $\wedge^cDX_t$. This
  proves sufficiency in the first part of
  Theorem~\ref{mthm:bivectparthyp}.

  For necessity, we just have to observe that domination of
  $\widetilde E\oplus\widetilde F$ by the action of
  $\wedge^cDX_t$ ensures (\ref{eq:sumliap}) holds for the
  Lyapunov spectrum of any given $X_t$-invariant
  probabililty measure $\mu$. Hence, in particular, we obtain
  \begin{align}\label{eq:diflyap}
    \lambda_s-\lambda_{s+1}=
    \lambda_s+\lambda_{s+2}+\dots+\lambda_{s+c}
    -
    (\lambda_{s+1}+\lambda_{s+2}+\dots+\lambda{s+c})
    < -\eta.
  \end{align}
  We now set $f_t(x)=\log\big(\|DX_t\mid
  E_x\|\cdot\|(DX_t\mid F_x)^{-1}\|\big)$ and, since we
  obtain~(\ref{eq:diflyap}) for an arbitrary $X_t$-invariant
  probability measure, we can apply again
  Lemma~\ref{le:sectional-arbieto} and deduce $f_t(x)\le
  \kappa -\eta t, t\ge0, x\in\Lambda$ for a constant
  $\kappa>0$, proving domination with respect to
  $DX_t$. This completes the proof of Theorem~\ref{mthm:bivectparthyp}.
\end{proof}

Now we prove Corollary~\ref{cor:bivectparthyp}. 

 \begin{proof}[Proof of Corollary~\ref{cor:bivectparthyp}]

  For the Corollary~\ref{cor:bivectparthyp},
  we assume that $T_\Gamma M$ admits a sectional hyperbolic
  splitting $E_\Gamma\oplus F_\Gamma$ with $\dim E_\Gamma=s$
  and $\dim F_\Gamma=c$. Then if $x\in\Gamma$ and
  $B=\{e_1,\dots,e_{c}\}$ is a basis for $F_x$, it is
  obvious that we can find $\lambda>0$ such that $\|DX_t
  u\wedge DX_t v\|\ge Ce^{\lambda t}$ for $t>0$ by
  definition of sectional hyperbolicity. Hence $\widetilde
  F=\wedge^2F$ is uniformly expanded by $\wedge^2 DX_t$.

  The reciprocal statement is straightforward. Indeed, let
  us assume that $T_\Gamma M$ admits a $DX_t$-invariant
  partial hyperbolic splitting $E\oplus F$ with $E$
  uniformly contracted, and $\wedge^2 T_\Gamma M$ admits a
  $\wedge^2 DX_t$-invariant and partial hyperbolic
  splitting $\widetilde E\oplus \widetilde F$ with
  $\widetilde F=\wedge^2 F$ and $\widetilde F$ uniformly
  expanded. Then clearly, given a basis $\{u,v\}$ of a
  two-dimensional subspace $G$ of $F$, we have that
  $\|\wedge^2DX_t\cdot(u\wedge v)\|$ grows exponentially,
  and this means that the area along $G$ is uniformly
  expanded. Hence $E\oplus F$ is a sectional hyperbolic
  splitting.

  This concludes the proof.
\end{proof}

\subsection{The three-dimensional case}
\label{sec:three-dimens-case}
\hfill

Here we prove Theorem~\ref{mthm:sec-exp-3d}.

Now $M$ is a $3$-manifold and $\Gamma$ is a compact
$X_t$-invariant subset having a singular-hyperbolic
splitting $T_\Gamma M=E_\Gamma\oplus F_\Gamma$.  By
Theorem~\ref{mthm:bivectparthyp} we have a
$\wedge^2DX_t$-invariant partial hyperbolic splitting
$\wedge^2 T_\Gamma M=\widetilde E\oplus \widetilde F$ with
$\dim\widetilde F=1$ and $\widetilde F$ uniformly
expanded. Following the proof of
Theorem~\ref{mthm:bivectparthyp}, if we write $e$ for a
unit vector in $E_x$ and $\{u,v\}$ an orthonormal base for
$F_x$, $x\in\Gamma$, then $\widetilde E_x$ is a
two-dimensional vector space spanned by $e\wedge u$ together
with $e\wedge v$.

From Theorem~\ref{mthm:Jseparated-parthyp} and the
existence of adapted metrics (see e.g.  \cite{Goum07}),
there exists a field $\J$ of quadratic forms so that
$X$ is $\J$-non-negative, $DX_t$ is strictly
$\J$-separated on a neighborhood $U$ of $\Gamma$,
$E_\Gamma$ is a negative subbundle, $F_\Gamma$ is a
positive subbundle and these subspaces are almost
orthogonal. In other words, there exists a function
$\delta:\Gamma\to\RR$ such that $\tilde\J_x-\delta(x)\J_x>0,
x\in\Gamma$ and we can locally write $\J(v)=\langle
J(v),v\rangle$ where $J=\diag\{-1,1,1\}$ with respect
to the basis $\{e,u,v\}$ and
$\langle\cdot,\cdot\rangle$ is the adapted inner
product; see \cite{arsal2012a}.

It is well-known that $A\wedge A=\det(A)\cdot (A^{-1})^*$
with respect to the adapted inner product which trivializes
$\J$, for any linear transformation $A:T_xM\to T_yM$. Hence
$\wedge^2DX_t(x)=\det (DX_t(x))\cdot (DX_{-t}\circ X_t)^*$
and a straightforward calculation shows that the
infinitesimal generator $D^2(x)$ of $\wedge^2DX_t$ equals
$\trace(DX(x))\cdot Id - DX(x)^*$.

Therefore, using the identification between $\wedge^2 T_xM$
and $T_xM$ through the adapted inner product and
Proposition~\ref{pr:J-separated-tildeJ}
\begin{align}
  \hat\J_x=\partial_t(-\J)(\wedge^2DX_t\cdot
  v)\mid_{t=0} &= \langle-(J\cdot D^2(x)+D^2(x)^*\cdot
  J)v,v\rangle\nonumber
  \\
  &= \langle [(\J\cdot
  DX(x)+DX(x)^*\cdot\J)-2\trace(DX(x))\J]v,v\rangle\nonumber
  \\
  &= (\tilde\J-2\trace(DX(x))\J)(v).\label{eq:square-generator}
\end{align}
To obtain strict $(-\J)$-separation of $\wedge^2DX_t$ we
search a function $\delta_2:\Gamma\to\RR$  so that
\begin{align*}
  (\tilde\J-2\trace(DX)\J) - \delta_2(-\J)>0
  \quad\text{or}\quad
  \tilde\J-(2\trace(DX)-\delta_2)\J>0.
\end{align*}
Hence it is enough to make $\delta_2=2\trace(DX)-\delta$.
This shows that in our setting $\wedge^2DX_t$ is
always strictly $(-\J)$-separated.

Finally, according to Theorem~\ref{pr:char-dom-split}, to
obtain the partial hyperbolic splitting of $\wedge^2DX_t$
which ensures singular-hyperbolicity, it is % necessary and
sufficient that either
$\widetilde\Delta_a^b(x)=\int_a^b\delta_2(X_s(x))\,ds$
satisfies item (1) of Proposition~\ref{pr:char-dom-split} or
$\hat\J_x$ is positive definite, for all
$x\in\Gamma$. This amounts precisely to the %necessary and
sufficient condition in the statement of
Theorem~\ref{mthm:sec-exp-3d} and we are done.

\subsection{Existence of adapted inner product for
  singular-hyperbolicity}
\label{sec:existence-adapted-in}

Now we show how to prove Theorem~\ref{mthm:adapted3d}
adapting the previous arguments.

From Theorem~\ref{mthm:bivectparthyp} we know that, for a
singular-hyperbolic attracting set $\Gamma$ for a
three-dimensional vector field with a splitting $E\oplus F$,
we have a partially hyperbolic splitting $\tilde E\oplus \tilde F$ for the
action of $\wedge^2DX_t$, where $F$ is uniformly expanded by
$\wedge^2DX_t$. Hence, from \cite[Theorem 1]{Goum07} , there
exists an adapted inner product $\langle\cdot,\cdot\rangle$
for this cocycle. Let $\|\cdot\|$ be the associated norm on
$T_\Gamma M$. Then there exists $\lambda>0$ such that
$\|(\wedge^2DX_t)\mid_{\tilde F}\|\ge e^{\lambda }$ for all $t>0$.

We know that $E,F$ are almost orthogonal with respect
to this inner product and we can choose a continuous
family of vectors $\{e_x\}$, a unit basis of $E_x$, and
$\{u_x,v_x\}$ an orthonormal basis of $F_x, x
\in\Gamma$. We define the linear operator
$J:T_xM\circlearrowleft$ in the basis $\{e_x,u_x,v_x\}$
such that its matrix is $\diag\{-1,1,1\}$. Now the
associated quadratic form $\J_x(w)=\langle
J(w),w\rangle$ is such that $\wedge^2DX_t$ is strictly
$(-\J)$-separated by construction; see \cite[Section
2.5]{arsal2012a}.

This means that there exists a continuous function
$\hat\delta:\Gamma\to\RR$ for which
$\hat\J-\hat\delta(-\J)>0$, where $\hat\J$
is given in (\ref{eq:square-generator}). That is, we have
$\tilde\J+(\hat\delta(x)-2\trace(DX(x)))\J>0$. Hence, if we
set $\delta(x)=\hat\delta(x)-2\trace(DX(x))$, then we
obtain strict $\J$-separation for $DX_t$ over $\Gamma$, as
guaranteed by Proposition~\ref{pr:J-separated-tildeJ}.

This ensures, in particular, that the norm
$|w|=\xi\sqrt{\J(w_E)^2+\J(w_F)^2}$ is adapted to the dominated splitting
$E\oplus F$ for the cocycle $DX_t$, where $w=w_E+w_F\in E_x\oplus
F_x, x\in\Gamma, \ \textrm{and} \ \xi$ is an
arbitrary positive constant; see \cite[Section
4.1]{arsal2012a}. This means that there exists $\mu>0$ such
that $|DX_t\mid_{E_x}|\cdot|DX_{-t}\mid_{F_{X_t(x)}}|\le e^{-\mu
  t}$ for all $t>0$.

Moreover, from the definition of the inner product and
the relation between $\wedge$ and the cross-product
$\times$, it follows that $|\det (DX_t\mid_{F_x})|=\|(\wedge^2DX_t)\cdot(u\wedge
v)\|=\|(\wedge^2DX_t)\mid_{F}\|\ge e^{\lambda t}$ for all
$t>0$, so $|\cdot|$ is adapted to the area expansion
along $F$.

To conclude, we are left to show that $E$ admits a
constant $\omega>0$ such that $|DX_t\mid_{E}|\le
e^{-\omega t}$ for all $t>0$. 

But since $E$ is
uniformly contracted, we know that $X(x)\in F_x$ for
all $x\in\Gamma$.

\begin{lemma}
  \label{le:flow-center}
  Let $\Gamma$ be a compact invariant set for a flow
  $X$ of a $C^1$ vector field $X$ on $M$.  Given a
  continuous splitting $T_\Gamma M = E\oplus F$ such
  that $E$ is uniformly contracted, then $X(x)\in F_x$
  for all $x\in \Lambda$.
\end{lemma}

\begin{proof}
See \cite[Lemma 5.1]{ararbsal2013} and \cite[Lemma 3.3]{arsal2012a}.
\end{proof}

On the one hand, on each non-singular point $x$ of
$\Gamma$ we obtain for each $w\in E_x$
  \begin{align*}
    e^{-\mu t}\ge \frac{|DX_t\cdot w|}{|DX_t\cdot X(x)|}
    =
    \frac{|DX_t\cdot w|}{|X(X_t(x))|}
    \ge
    \frac{|DX_t\cdot v|}{\sup\{|X(z)|:z\in\Gamma\}}
    \ge
    |DX_t\cdot v|,
  \end{align*}
  since we can always choose a small enough constant $\xi>0$
  in such a way that $\sup\{|X(z)|:z\in\Gamma\}\le1$. We
  note that the choice of the positive constant $\xi$ does
  not change any of the previous relations involving
  $|\cdot|$.

  On the other hand, for $\sigma\in\Gamma$ such that
  $X(\sigma)=0$, we fix $t>0$ and, since $\Gamma$
  is a non-trivial invariant set, we can find a
  sequence $x_n\to\sigma$ of regular points of
  $\Gamma$. The continuity of the derivative cocycle
  ensures $|DX_t\mid_{E_\sigma}|=\lim_{n\to\infty}|DX_t\mid_{E_{x_n}}|\le
  e^{-\lambda t}$. Since $t>0$ was arbitrarily chosen,
  we see that $|\cdot|$ is adapted for the contraction
  along $E_\sigma$.

  This shows that $\omega=\mu$ and completes the proof of
  Theorem~\ref{mthm:adapted3d}.

\section{Examples of application}
\label{sec:comments-organiz-tex}

We present some examples showing that the statement of
Theorem~\ref{mthm:bivectparthyp} does not extend easily
to a higher dimensional setting.

\begin{example}\label{ex:no-form}
  In a higher dimensional setting, consider $\sigma$ a
  hyperbolic fixed point for a vector field $X$ in a
  $4$-manifold such that $DX(\sigma)$ is diagonal with
  eigenvalues $\lambda_0<\lambda_1<\lambda_2<0<\lambda_3$
  along the coordinate axis, satisfying
  $\lambda_1+\lambda_3>0$ (this is similar to the Lorenz
  singularity except for the extra contracting direction
  corresponding to $\lambda_0$). Consider also the quadratic
  form $\J(\vec x)= -x_0^2-x_1^2+x_2^2+x_3^2=\langle J\vec
  x,\vec x\rangle$ with $J=\diag\{-1,-1,1,1\}$ on $T_\sigma
  M$.  It is standard to define a bilinear form on
  $\wedge^2T_\sigma M$ using $\J$ by
  \begin{align}\label{eq:innerproductbivectors}
    (u_1\wedge u_2, v_1\wedge v_2)
    =
    \det
    \begin{pmatrix}
      \langle Ju_1, v_1\rangle & \langle Ju_1, v_2\rangle
      \\
      \langle Ju_2, v_1\rangle &\langle Ju_2, v_2\rangle
    \end{pmatrix}
  \end{align}
on simple bivectors and then extend by linearity to the
whole $\wedge^2T_\sigma M$.

However, letting $e_0,e_1,e_2,e_3$ be the canonical base,
$(e_i\wedge e_j,e_i\wedge e_j)=-1, i=0,1, j=2,3;$ but $(e_1\wedge
e_2,e_1\wedge e_2)=1=(e_2\wedge e_3,e_2\wedge e_3)$, and
$e_1\wedge e_2$ is contracted while $e_2\wedge e_3$ is
expanded by $\wedge^2DX_t=\wedge^2e^{tDX(\sigma)}$;
likewise $e_1\wedge e_2$ is contracted but $e_1\wedge e_3$
is expanded. Thus we have mixed behavior with both positive
and negative bivectors.

Hence, the standard way of building a quadratic form on
$\wedge^2 T_\sigma M$ \emph{from a quadratic form on
  $T_\sigma M$} does not capture the the partial hyperbolic
behavior on bivectors.
\end{example}

In the above example, the problem was caused by the
increased dimension of the negative $\J$-subspace, as the
following example shows.

\begin{example}
  \label{ex:Jcodim1}
  The case of codimension one: let us assume that $E\oplus
  F$ is a sectional-hyperbolic splitting over a compact
  invariant subset $\Gamma$ of a $C^1$ vector field $X$,
  where $E$ is one-dimensional and $F$ has arbitrary
  dimension. Then we have strict $\J$-separation for a
  certain family of quadratic forms which are given by
  $\J(u)=\langle J(u),u\rangle, u\in T_\Gamma M$ for a
  certain non-singular linear operator $J=J_x:T_x
  M\circlearrowleft$.

  We can now define an bilinear form on $\wedge^2 T_\Gamma
  M$ as in (\ref{eq:innerproductbivectors}) and check that,
  due to the one-dimensional character of $E$, the new form
  gives positive values to bivectors $u\wedge v$ where
  \emph{both $u,v$ belong to $F$}; and negative values to
  bivectors $u\wedge v$ where \emph{only $u$ belongs to
    $E$}. These two classes of bivectors split $\wedge^2
  T_\Gamma M=\widetilde E\oplus \widetilde F$ as in the
  statement of Theorem~\ref{mthm:bivectparthyp}.

  Hence, in this codimension one setting, we may use the
  standard construction of a bilinear form on the external
  square of a vector space to obtain a quadratic form which
  is suitable to study domination and partial hyperbolicity,
  directly from the originally given $\J$-separating
  quadratic form.
\end{example}

\begin{example}\label{ex:ThmA}
  Theorem~\ref{mthm:bivectparthyp} does not hold if we take
  $c < \dim F$: consider $\sigma$ a hyperbolic fixed point for
  a vector field $X$ in a $4$-manifold such that
  $DX(\sigma)=\diag\{-2, 1, 3, 10\}$. The splitting
  $E=\RR\times\{0^3\}, F=\{0\}\times\RR^3$ is dominated and
  hyperbolic but, for $c=2<3=\dim F$ the splitting $\tilde
  E\oplus \tilde F$ of the exterior square is not
  dominated. Indeed, the eigenvalues for $\tilde F$ are $1+3
  = 4, 1+10 = 11, 3+10 =13$, and for $\tilde E $ the
  eigenvalues are $-2+1 = -1, -2+3 = 1, -2+10 = 8$, so we
  have an eigenvalue $8$ in $\tilde E$ striclty bigger than
  the eigenvalue $4$ along $\tilde F$.
\end{example}

We now present applications of these results. First a very
simple but illustrative example.

\begin{example}\label{ex:Lorenz-like-J}
  Let us consider a hyperbolic saddle singularity $\sigma$
  at the origin for a smooth vector field $X$ on $\RR^3$
  such that the eigenvalues of $DX(\sigma)$ are real and
  satisfy $\lambda_2<\lambda_3<0<\lambda_1$. Through a
  coordinate change, we may assume that
  $D=DX(\sigma)=\diag\{\lambda_1,\lambda_2,\lambda_3\}$ and
  $DX_t(\sigma)=e^{t D}=\diag\{e^{\lambda_1 t},e^{\lambda_2
    t},e^{\lambda_3 t}\}$, $\wedge^2DX_t(\sigma)=
  % \diag\{e^{(\lambda_2+\lambda_3)t},e^{(\lambda_1+\lambda_3)t},e^{(\lambda_1+\lambda_2)t}\}
  % =
  e^{t  D^2}$ where $D^2=D^2(\sigma)=\trace(DX(\sigma))\cdot
  Id-DX(\sigma)^*=\diag\{\lambda_2+\lambda_3,
  \lambda_1+\lambda_3, \lambda_1+\lambda_2\}$.

  We take the quadratic form $\J(x,y,z)=x^2-y^2+z^2$ in
  $\RR^3$. Then $\J$ is represented by the matrix
  $J=\diag\{1,-1,1\}$, that is, $\J(w)=\langle
  J(w), w\rangle$ with the canonical inner
  product.

  Then $\tilde J=J\cdot D+D^*\cdot
  J=\diag\{2\lambda_1,-2\lambda_2,2\lambda_3\}$ and
  $\tilde J-\delta J>0\iff 2\lambda_2<\delta<2\lambda_3<0$. So
  $\delta$ is negative. From Proposition~\ref{pr:J-separated} and
  Proposition~\ref{pr:J-separated-tildeJ} we have strict
  $\J$-separation, thus partial hyperbolicity with the
  $y$-axis a uniformly contracted direction
  ($\J$-negative) dominated by the $yz$-direction
  ($\J$-positive).

  This is not a hyperbolic splitting and the conclusion
  would be the same if $\lambda_1$ where negative: we would
  get a sink with a partially hyperbolic splitting. Note
  also that $\Delta_0^t$ satisfies item (2) of
  Proposition~\ref{pr:char-dom-split}.

  As for $(-\J)$-separation of the exterior square, we have
  $\delta_2=2\trace(D)-\delta=
  2(\lambda_1+\lambda_2+\lambda_3)-\delta$ so
  $2(\lambda_1+\lambda_2)<\delta_2<2(\lambda_1+\lambda_3)$. Hence
  we can take $\delta_2>0$ if $\lambda_1+\lambda_3>0$ (note
  that $\lambda_1+\lambda_2>0$ implies this relation) and
  $\delta\gtrapprox2\lambda_2$. The condition
  $\lambda_1+\lambda_3>0$ is precisely the ``Lorenz-like''
  condition satisfied by the singularities of
  singular-hyperbolic attractors; see
  e.g. \cite{AraPac2010}.

  In this setting, we have both strict $\J$-separation and
  sectional-expansion along the $yz$-direction for $X_t$
  (and so singular-hyperbolicity), since $\delta_2>0$
  ensures that condition (1) in
  Theorem~\ref{mthm:sec-exp-3d} is true.
\end{example}

Now we show how our results simplify the checking of
singular-hyperbolicity in the standard example of a
geometric Lorenz flow.

\begin{example}
  \label{ex:geomLorenz}
  We now consider the geometric Lorenz example as
  constructed in \cite[Chapter 3, Section 3]{AraPac2010};
  see Figure~\ref{fig:Lorenz}.
  \begin{figure}[htpb]
    \centering
    \includegraphics[width=6cm]{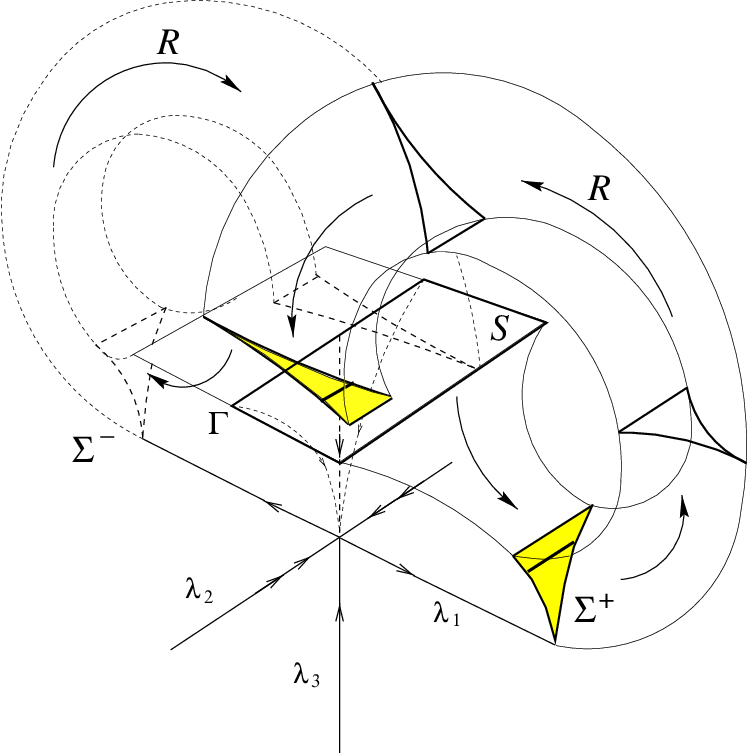}
    \caption{The geometric Lorenz flow.}
    \label{fig:Lorenz}
  \end{figure}
  In the linear region around the origin we have
  $X(\xi)=D\cdot\xi$ and so $DX(\xi)=D$ for all points
  $\xi$ around the origin, where $D$ was defined in
  Example~\ref{ex:Lorenz-like-J}. Hence we can
  calculate, as in Example~\ref{ex:Lorenz-like-J} with
  $\lambda_1+\lambda_3>0$, and show that we have strict
  $\J$ separation with
  $2\lambda_2<\delta<2\lambda_3<0$, that is, $\delta$
  is negative and bounded away from zero.

  In the same region we also get strict $(-\J)$-separation
  for the exterior square of the cocycle $DX_t$ with
  $\delta_2=2\trace(D)-\delta$ between
  $2(\lambda_1+\lambda_2)$ and $2(\lambda_1+\lambda_3)$.
  Thus, $\delta_2$ can be taken positive and bounded away
  from zero by setting $\delta\gtrapprox 2\lambda_2$.

  We are left to prove these properties in the lobes, where
  the flow is a combination of a rotation on the $xz$-plane,
  a dilation and a translation on the $y$-direction; and
  then check the singular-hyperbolic conditions on the
  transitional region between the lobes and the linear flow.

  We can write the vector field in the interior of the lobes
  as $X_i'=A_i\cdot(X_i-C_i)+P_i$, where $C_i$ is the center
  of the rotation, $P_i$ is a vector representing a
  translation and
  \begin{align*}
    A_i=
    \begin{pmatrix}
      \varrho\lambda_1 & 0 & -(-1)^i \\
      0 &  \zeta\lambda_2 & 0 \\
      (-1)^i &  0 & \varrho\lambda_1
    \end{pmatrix}
    \quad\text{with}\quad
    0<\varrho,\zeta\ll 1, i=1,2.
  \end{align*}
  Here $i=1$ corresponds to the lobe starting with $x>1$ and
  $i=2$ to the other lobe.  We observe that by an affine
  change of coordinates we can write the vector field as
  $Y'=A_i\cdot Y$.

  Using the same quadratic form $\J$ we get $\tilde
  J_i=J\cdot A_i+A_i^*\cdot J=
  \diag\{2\varrho\lambda_1,-2\zeta\lambda_2,2\varrho\lambda_1\}$
  , $i=1,2$.

  For the exterior square of derivative cocycle of the flow,
  we observe that $\delta_2$ becomes
  $2(2\varrho\lambda_1-\zeta\lambda_2)-\delta$ and
  $2\varrho\lambda_1-\zeta\lambda_2 >0$.

  Hence, there exists $\delta_1>0$ such that for all
  $\delta<\delta_1$ we have $\tilde \J_i-\delta\J>0$ and
  $\delta_2>0$ on the lobes. So, singular-hyperbolicity is
  also verified on the lobes.

  We still have to check the transition between the linear
  region and the lobes.

  We can find a smooth path $\tilde A$ from $D$ in the
  linear region to $A_i$ in the lobes made of symmetric
  matrices, ensuring that $\tilde J$ will remain diagonal:
  it is enough to take the line segment between $D$ and
  $A_i$ in $\RR^{3\times3}$ and in this way the signs of the
  first and second elements of the diagonal of the
  corresponding matrix $\tilde J$ do not change. However,
  the last element of the diagonal goes from $2\lambda_3<0$
  to $2\varrho\lambda_1>0$.

  Therefore, since the vector field $Z_i$ in the
  transitional region is defined as a linear combination
  $\mu X+(1-\mu)X_i$ of the fields in the linear region and
  the lobes, $i=1,2$ and $0\le\mu\le1$, we can write $\tilde
  J$ in the transitional region as
  \begin{align*}
    \diag\{
    2\lambda_1(\mu+\rho(1-\mu)), -2\lambda_2(\mu+\zeta(1-\mu)),
    2(\lambda_3\mu+\rho\lambda_1(1-\mu))\}.
  \end{align*}
  Hence $\tilde J-\delta J>0$ subject to the condition
  $2\lambda_2<\delta<2\lambda_3$, as in
  Example~\ref{ex:Lorenz-like-J}. This provides partial
  hyperbolicity with the negative $\J$-cone uniformly
  contracting on the geometric Lorenz attractor.

  Finally, $\delta_2=2\trace(DZ_i)-\delta=2\trace(D)\cdot\mu
  +2\trace(A_i)\cdot(1-\mu)-\delta$ and so, if we take
  $2\lambda_2\lessapprox\delta<2\lambda_3<0$, then we again
  obtain $\delta_2>0$, because the condition on $\delta$ is
  compatible with all the previous conditions on the linear
  region and on the lobes.

  In this way we have strictly negative $\delta$ and
  strictly positive $\delta_2$ on all points in a
  neighborhood of the geometric Lorenz attractor with
  respect to $\J$ and, from Theorem~\ref{mthm:sec-exp-3d},
  this alone ensures that the geometric Lorenz attractor is
  a singular-hyperbolic set.
\end{example}

%%%%%%%%%%%%%%%%%%%%%%%%%%%%%%%%%%%

\def\cprime{$'$}

% \bibliographystyle{abbrv}
% \bibliography{../bibliobase/bibliography}

\begin{thebibliography}{10}

\bibitem{Aleks68}
V.~M. Alekseev.
\newblock Quasirandom dynamical systems. {I}. {Q}uasirandom diffeomorphisms.
\newblock {\em Mat. Sb. (N.S.)}, 76 (118):72--134, 1968.

\bibitem{Aleks68a}
V.~M. Alekseev.
\newblock Quasirandom dynamical systems. {II}. {O}ne-dimensional nonlinear
  vibrations in a periodically perturbed field.
\newblock {\em Mat. Sb. (N.S.)}, 77 (119):545--601, 1968.

\bibitem{Aleks69}
V.~M. Alekseev.
\newblock Quasirandom dynamical systems. {III}. {Q}uasirandom vibrations of
  one-dimensional oscillators.
\newblock {\em Mat. Sb. (N.S.)}, 78 (120):3--50, 1969.

\bibitem{AraPac2010}
V.~Ara{\'u}jo and M.~J. Pacifico.
\newblock {\em Three-dimensional flows}, volume~53 of {\em Ergebnisse der
  Mathematik und ihrer Grenzgebiete. 3. Folge. A Series of Modern Surveys in
  Mathematics [Results in Mathematics and Related Areas. 3rd Series. A Series
  of Modern Surveys in Mathematics]}.
\newblock Springer, Heidelberg, 2010.
\newblock With a foreword by Marcelo Viana.

\bibitem{ararbsal2013}
V.~Araujo, A.~Arbieto and L.~Salgado.
\newblock Dominated splittings for flows with singularities.
\newblock {\em Nonlinearity} 26, 2391--2407. 2013.

\bibitem{arsal2012a}
V.~Araujo and L.~Salgado.
\newblock Infinitesimal lyapunov functions for singular flows.
\newblock {\em Mathematische Zeitschrift (online)}, pages 1--35, 2013.

\bibitem{arbieto2010}
A.~Arbieto.
\newblock Sectional lyapunov exponents.
\newblock {\em Proc. of the Amercian Mathematical Society}, 138:3171--3178,
  2010.

\bibitem{ArbSal2011}
A.~Arbieto, L.~Salgado, On critical orbits and sectional hyperbolicity of the nonwandering set for flows,
 \emph{Journal of differential equations} \textbf{250} (2011) 2927--2939.

\bibitem{arnold-l-1998}
L.~Arnold.
\newblock {\em {Random dynamical systems}}.
\newblock {Springer-Verlag}, {Berlin}, {1998}.

\bibitem{BDV2004}
C.~Bonatti, L.~J. D{\'i}az, and M.~Viana.
\newblock {\em {Dynamics beyond uniform hyperbolicity}}, volume {102} of {\em
  {Encyclopaedia of Mathematical Sciences}}.
\newblock {Springer-Verlag}, {Berlin}, {2005}.
\newblock {A global geometric and probabilistic perspective, Mathematical
  Physics, III}.

\bibitem{Bo75}
R.~Bowen.
\newblock {\em {Equilibrium states and the ergodic theory of Anosov
  diffeomorphisms}}, volume {470} of {\em {Lect. Notes in Math.}}
\newblock {Springer Verlag}, {1975}.

\bibitem{BR75}
R.~Bowen and D.~Ruelle.
\newblock {The ergodic theory of Axiom A flows}.
\newblock {\em {Invent. Math.}}, {29}:{181--202}, {1975}.

\bibitem{Goum07}
N.~Gourmelon.
\newblock Adapted metrics for dominated splittings.
\newblock {\em Ergodic Theory Dynam. Systems}, 27(6):1839--1849, 2007.

\bibitem{HuntMack03}
T.~J. Hunt and R.~S. MacKay.
\newblock Anosov parameter values for the triple linkage and a physical system
  with a uniformly chaotic attractor.
\newblock {\em Nonlinearity}, 16(4):1499, 2003.

\bibitem{BurnKatok94}
A.~Katok.
\newblock Infinitesimal {L}yapunov functions, invariant cone families and
  stochastic properties of smooth dynamical systems.
\newblock {\em Ergodic Theory Dynam. Systems}, 14(4):757--785, 1994.
\newblock With the collaboration of Keith Burns.

\bibitem{KH95}
A.~Katok and B.~Hasselblatt.
\newblock {\em {Introduction to the modern theory of dynamical systems}},
  volume~{54} of {\em {Encyclopeadia Appl. Math.}}
\newblock {Cambridge University Press}, {Cambridge}, {1995}.

\bibitem{lewow89}
J.~Lewowicz.
\newblock Expansive homeomorphisms of surfaces.
\newblock {\em Bol. Soc. Brasil. Mat. (N.S.)}, 20(1):113--133, 1989.

\bibitem{MeMor06}
R.~Metzger and C.~Morales.
\newblock Sectional-hyperbolic systems.
\newblock {\em Ergodic Theory and Dynamical System}, 28:1587--1597, 2008.

\bibitem{MPP99}
C.~A. Morales, M.~J. Pacifico, and E.~R. Pujals.
\newblock Singular hyperbolic systems.
\newblock {\em Proc. Amer. Math. Soc.}, 127(11):3393--3401, 1999.

\bibitem{MPP04}
C.~A. Morales, M.~J. Pacifico, and E.~R. Pujals.
\newblock {Robust transitive singular sets for 3-flows are partially hyperbolic
  attractors or repellers}.
\newblock {\em {Ann. of Math. (2)}}, {160}({2}):{375--432}, {2004}.

\bibitem{Newhouse2004}
S.~Newhouse.
\newblock Cone-fields, domination, and hyperbolicity.
\newblock In {\em Modern dynamical systems and applications}, pages 419--432.
  Cambridge Univ. Press, Cambridge, 2004.

\bibitem{robinson1999}
C.~Robinson.
\newblock {\em {Dynamical systems}}.
\newblock {Studies in Advanced Mathematics}. {CRC Press}, {Boca Raton, FL},
  {Second} edition, {1999}.
\newblock {Stability, symbolic dynamics, and chaos}.

\bibitem{robinson2004}
C.~Robinson.
\newblock {\em {An introduction to dynamical systems: continuous and
  discrete}}.
\newblock {Pearson Prentice Hall, Upper Saddle River, NJ}, {2004}.

\bibitem{luciana-tese}
L.~Salgado.
\newblock {\em Sobre hiperbolicidade fraca para fluxos singulares}.
\newblock PhD thesis, UFRJ, Rio de Janeiro, {2012}.

\bibitem{Sh87}
M.~Shub.
\newblock {\em {Global stability of dynamical systems}}.
\newblock {Springer Verlag}, {1987}.

\bibitem{Tu99}
W.~Tucker.
\newblock {The Lorenz attractor exists}.
\newblock {\em {C. R. Acad. Sci. Paris}}, {328, S{\'e}rie I}:{1197--1202},
  {1999}.

\bibitem{viana2000i}
M.~Viana.
\newblock {What{'}s new on Lorenz strange attractor}.
\newblock {\em {Mathematical Intelligencer}}, {22}({3}):{6--19}, {2000}.

\bibitem{Winitzki12}
S.~Winitzki.
\newblock {\em Linear Algebra: Via Exterior Products}.
\newblock Lulu.com, Raleigh, N.C., USA, 1.2 edition, 2012.

\bibitem{Wojtk85}
M.~Wojtkowski.
\newblock Invariant families of cones and {L}yapunov exponents.
\newblock {\em Ergodic Theory Dynam. Systems}, 5(1):145--161, 1985.
\end{thebibliography}

\end{document}